%% file: sknr_arxiv.tex
\documentclass[oneside]{amsart}
\usepackage[foot]{amsaddr}

\usepackage[round]{natbib}

\usepackage[utf8]{inputenc} % allow utf-8 input
\usepackage[T1]{fontenc}    % use 8-bit T1 fonts
\usepackage{url}            % simple URL typesetting
\usepackage{booktabs}       % professional-quality tables
\usepackage{amsfonts}       % blackboard math symbols
\usepackage{nicefrac}       % compact symbols for 1/2, etc.
\usepackage{microtype}      % microtypography

\usepackage[dvipsnames,svgnames]{xcolor}		% for color
\colorlet{MyRed}{Crimson!75!black}
\colorlet{MyGreen}{DarkGreen!80!black}
\colorlet{MyBlue}{MediumBlue!80!black}

\usepackage{hyperref}
\hypersetup{
colorlinks=true,
linktocpage=true,
pdfstartview=FitH,
breaklinks=true,
pdfpagemode=UseNone,
pageanchor=true,
pdfpagemode=UseOutlines,
plainpages=false,
bookmarksnumbered,
bookmarksopen=false,
bookmarksopenlevel=1,
hypertexnames=true,
pdfhighlight=/O,
urlcolor=MyBlue,linkcolor=MyBlue,citecolor=MyBlue,
pdftitle={},
pdfauthor={},
pdfsubject={},
pdfkeywords={},
pdfcreator={pdfLaTeX},
pdfproducer={LaTeX with hyperref}
}

\makeatletter
\renewcommand\paragraph{\@startsection{paragraph}{4}{\z@}%
  {-3.25ex\@plus -1ex \@minus -.2ex}%
  {-1em}%
  {\normalfont\normalsize\bfseries}}
\makeatother

\title[EOT via Stable Low Frequency Modes]{Faster Computation of Entropic Optimal Transport via Stable Low Frequency Modes}

\author[R. Chhaibi]{Reda Chhaibi$^{\ast}$}
\address{$^{\ast}$Université Côte d'Azur, CNRS, LJAD, Parc Valrose, 06108 Nice Cedex 02, France}
\email{reda.chhaibi@univ-cotedazur.fr}

\author[S. Gratton]{Serge Gratton$^{\dagger}$}
\address{$^{\dagger}$Toulouse INP-ENSEEIHT, Institut de Recherche en Informatique de Toulouse (IRIT), 2 rue Camichel, BP 7122, 31071 Toulouse Cedex 7, France}
\email{serge.gratton@univ-toulouse.fr}

\author[S. Vaiter]{Samuel Vaiter$^{\ast}$}
\email{samuel.vaiter@cnrs.fr}

\thanks{RC acknowledges support from ANR POAS (ANR-24-CE40-5511) and Institut Universitaire de France (IUF). SV acknowledges support from ANR MAD (ANR-24-CE23-1529), PEPR PDE-AI (ANR-23-PEIA-0004) and the chair 3IA BOGL (ANR-23-IACL-0001).}

\input{preamble}

\begin{document}

\maketitle

\begin{abstract}
  In this paper, we propose an accelerated version for the Sinkhorn algorithm, which is the reference method for computing the solution to Entropic Optimal Transport. 
  Its main draw-back is the exponential slow-down of convergence as the regularization weakens $\varepsilon \rightarrow 0$.
  Thanks to spectral insights on the behavior of the Hessian, we propose to mitigate the problem via an original spectral warm-start strategy. This leads to faster convergence compared to the reference method, as also demonstrated in our numerical experiments.
\end{abstract}

\input{content}

\bibliography{EOT}
\bibliographystyle{abbrvnat}

\input{appendix}

\end{document}

%% file: preamble.tex
\usepackage{amsmath,amsfonts,amssymb,amsopn,amscd,amsthm}
\usepackage{gensymb}
\usepackage{comment}
\usepackage{dsfont}
\usepackage{algorithm}
\usepackage{algpseudocode}
\usepackage{bm}
\usepackage{caption, graphicx}
\usepackage{color}
\usepackage{enumitem}
\usepackage{makecell}
\usepackage{multirow}
\usepackage{todonotes}
\usepackage{wrapfig}

\def\1{{\mathbf 1}}

\def\Op{\mathrm{Op}}
\def\Id{\mathrm{Id}}

\def\R{{\mathbb R}}

\def\K{{\mathbb K}}

\def\Cc{{\mathcal C}}

\def\Lc{{\mathcal L}}
\def\Mc{{\mathcal M}}

\def\Oc{{\mathcal O}}

\def\Xc{{\mathcal X}}
\def\Yc{{\mathcal Y}}

\newtheorem{thm}{Theorem}[section]

\newtheorem{lemma}[thm]{Lemma}

\newtheorem{rmk}[thm]{Remark}

\numberwithin{algorithm}{section}
\numberwithin{equation}{section}
\numberwithin{figure}{section}
\numberwithin{table}{section}

%% file: content.tex
% --------------------------------------------------
\section{Introduction}
\begin{wrapfigure}{r}{0.3\textwidth}
  \begin{center}
    \includegraphics[width=0.3\textwidth]{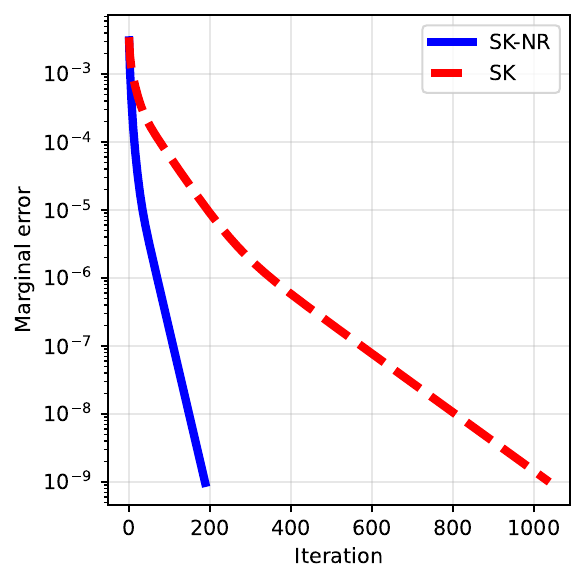}
  \end{center}
  \caption{SK-NR is a very efficient warmstart strategy for decreasing value of $\varepsilon$.}
\end{wrapfigure}

Optimal transport (OT) equips the space $\mathcal{M}_1(\mathcal{X})$ of probability measures with the Wasserstein -- or Earth Mover’s -- distance, a metric that captures its geometric structure.
This feature have inspired substantial mathematical research; see the monographs \citep{villani2008optimal, villani2021topics} for detailed treatments.
On the computational side, the entropic regularization of optimal transport (EOT) advocated by \citet{cuturi2013sinkhorn}, solved via the Sinkhorn--Knopp (SK) iteration, serves as a foundation for large‑scale machine learning applications.
The survey of \citet{peyre2019computational} provides an in-depth overview.

The SK algorithm, aka the iterative proportional fitting algorithm, is a fixed-point algorithm with a contraction coefficient $q^* = q^*(\varepsilon)$. Following \citet[Proposition 19]{vialard2019elementary}, there is a constant $\kappa>0$, depending only on the measures and the cost, such that for all $\varepsilon>0$,
\begin{align}
\label{eq:contraction_coefficient}
   q^*(\varepsilon) \leq & 1 - \exp\left( - \frac{\kappa}{\varepsilon} \right)\ .
\end{align}
Numerical experiments suggest this upper bound is tight -- See \cite[Fig. 4.5]{peyre2019computational}.
Hence, SK and their variants suffer from slow convergence as $\varepsilon$ decreases, a key issue we aim to address in this work.

\paragraph{Related work.} While entropic regularization has made OT scalable for machine-learning tasks with a sufficiently large temperature~$\varepsilon$. Yet, when the downstream task requires couplings, e.g., recovering transport maps for lineage or trajectory inference~\citep{schiebinger2019optimal,lavenant2024toward}, one must drive $\varepsilon$ toward zero, at which point the vanilla Sinkhorn iteration becomes prohibitively slow.
Several works modify the fixed-point iteration itself. \citet{NEURIPS202017f98ddf} show that the Sinkhorn divergence is a nearly unbiased plug-in estimator of the squared Wasserstein distance and, because it tolerates larger regularization with faster effective convergence to the unregularized cost.
Richardson-style extrapolation in~\citep{a14050143} reduces the bias by combining iterates at geometrically scaled temperatures.
From an algorithmic perspective, \citet{tang2024accelerating} use incomplete Sinkhorn sweeps with sparse Newton updates, while \citet{lin2022efficiency} use the dual ascent with a Nesterov acceleration.
Viewing Sinkhorn as coordinate descent, \citet{altschuler2017near} propose Greenkhorn, a greedy variant that updates the most violated marginal at each step, together with its stochastic extension~\citep{abid2018stochastic}.
Nyström kernel compression alleviates the quadratic memory footprint in high dimensions~\citep{altschuler2019massively}.
Finally, the Screenkhorn algorithm identifies inactive dual variables and screens them out before iteration, leading to empirical speed-ups~\citep{alaya2019screening}.

Instead of alternating projections, one may solve the dual of entropic OT with first-order methods.
The stochastic averaged gradient scheme of \citet{genevay2016stochastic} enjoys memory linear in the minibatch size, while the primal-dual mirror-prox strategy of \citet{pmlr-v80-dvurechensky18a} attains a complexity that matches the best Sinkhorn bounds up to logarithmic factors.

A complementary line of work exploits the smoothness of the regularized objective with respect to~$\varepsilon$.
Warm-starting Sinkhorn with potentials computed at a nearby temperature can slash iteration counts, and recent analyses clarify how to choose those initializations~\citep{thornton2023rethinking}.
Stability results in~\citep{schmitzer2019stabilized} justify partially such homotopy strategies.
The annealed Sinkhorn schedule~\citep{chizat2024annealedsinkhornoptimaltransport} with the early proposal of~\citep{kosowsky1991solving}, formalizes this idea: it provably preserves linear convergence while adaptively lowering~$\varepsilon$ along the trajectory.

\paragraph{Contributions.}
Our contributions are as follows.
\begin{enumerate}
\item \textbf{Spectral behavior.}
We prove a \emph{structure theorem} (Theorem~\ref{thm:hessian_structure}) showing that the dual Hessian decomposes as
$\nabla^{2}Q_{\varepsilon}(f^\ast,g^\ast)=-\frac{1}{\varepsilon}(\mathrm{Id}+\mathbb{K}_{\varepsilon})$
where the spectrum of the integral operator $\mathbb{K}_{\varepsilon}$ is confined to $[-1,1]$, potentially clustering at the endpoints.
Consequently, only a handful of ``low-frequency'' eigendirections become poorly conditioned when $\varepsilon \to 0$.

\item \textbf{Mixed first/second-order solver (SK–NR$(\ell)$).}
Exploiting this insight, we design an algorithm that alternates standard Sinkhorn–Knopp updates with
\emph{Newton–Raphson steps restricted to the $\ell$ unstable directions}.
The remaining coordinates are still treated by inexpensive SK rescalings, so each iteration costs
$O(nm + \ell^{3})$, essentially the same as SK when $\ell \ll (nm)^{1/3}$.

\item \textbf{Dimension-free acceleration guarantees.}
We prove global convergence (Theorem~\ref{thm:convergence}) of {SK–NR}$(\ell)$ and show that its linear rate is governed by the $(\ell+1)$-st eigenvalue of $\K_{\varepsilon}$
$$\|f_{k+1}-f^\ast\| \;\le\; \rho_{\ell+1}^{\,2}\,\|f_k-f^\ast\| \;+\; C\,\|f_k-f^\ast\|^{2},$$
so choosing even a small $\ell$ exponentially reduces the iteration count as $\varepsilon$ decreases.

\item \textbf{Practical speed-ups at low regularisation.}
On synthetic 2-D point clouds and a single-cell gene-expression OT task,
{SK–NR}$(\ell)$ delivers order-of-magnitude reductions in iterations compared with SK,
especially in the small-$\varepsilon$ regime relevant for map recovery.
\end{enumerate}

\paragraph{Notations.} For a given Polish space $E$, the convex set of probability measures is written $\Mc_1(E)$. Positive Radon measures are written $\Mc_+(E)$. Given a bounded function $f$ and a measure $\alpha \in \Mc_1(E)$, the duality bracket is $\langle f, \alpha \rangle$.

% --------------------------------------------------
\section{Preliminaries on Entropic Optimal Transport}
\label{section:theory}

\paragraph{Kantorovich formulation, primal and dual.}
OT is the general problem of moving one distribution of mass to another  as efficiently as possible. For the purposes of this paper, we refrain from the measure-theoretic setting required by general Polish spaces. We consider two discrete spaces, usually samples, 
\begin{align}
\label{def:spaces}
\Xc = \left\{ x_1, x_2, \dots, x_n \right\} , \ & 
\quad \quad
\Yc = \left\{ y_1, y_2, \dots, y_m \right\} . \ 
\end{align}
These two spaces are endowed with reference probability measures $\alpha \in \Mc_1(\Xc)$ and $\beta \in \Mc_1(\Yc)$.
A cost function $\Cc: \Xc \times \Yc \rightarrow \R$ gives the cost of moving $x_i$ to $y_j$. Because of the discrete nature of the spaces $\mathcal{X}$ and $\mathcal{Y}$ \eqref{def:spaces}, it is identified to a matrix $\Cc \in \R_+^{n \times m}$.

Consider the set of all couplings, with marginals $\alpha \in \Mc_1(\mathcal{X})$ and $\beta \in \Mc_1(\mathcal{Y})$, as
$$ \Pi(\alpha,\beta)
   :=
   \left\{
   \pi \in \Mc_+(\mathcal{X} \times \mathcal{Y}) \approx \R_+^{n\times m} \ | \ \pi \mathds{1}_{m}=\alpha, \pi^{T}\mathds{1}_{n}=\beta\right\} \ .
$$
The Kantorovich setup for discrete spaces written is known as the Wasserstein distance $W^{\Cc}$ defined as
\begin{align}
W^{\Cc}(\alpha,\beta) := & \min_{\pi \in \Pi(\alpha,\beta)}\langle \Cc, \pi \rangle
\label{eq:OT_primal}
\\
 = & \sup_{f \oplus g \leq \Cc} \langle f, \alpha \rangle + \langle g, \beta \rangle
 \ .
 \label{eq:OT_dual}
\end{align}
Here the notation $f \oplus g$ refers to the two-variable function $f \oplus g: (x,y) \mapsto f(x) + g(y)$.
Eq. \eqref{eq:OT_primal} and Eq. \eqref{eq:OT_dual} are respectively identified as the primal and dual formulations of the Kantorovich problem.
For an account of the equivalence between primal and dual formulation, using the Fenchel-Rockafellar convex duality, under mild assumptions on the cost function $\Cc$, we refer to \citep{villani2021topics}. The pair of functions $(f,g)$ is commonly referred to as Kantorovich potentials.

\paragraph{Entropic regularization of OT (EOT).}
The OT problems given in Eq. \eqref{eq:OT_primal} and Eq. \eqref{eq:OT_dual} are linear problems under convex contrains. For better properties, the idea is to add a convex penalization term weighted by a penalization parameter $\varepsilon>0$. The choice of the Kullback-Leibler divergence
$KL( \pi \|\alpha \otimes \beta) := \int_{\Xc \times \Yc}\log\left(\frac{d\pi}{d\alpha \otimes \beta}(x,y)\right)\pi(dxdy)$
is a particularly tractable choice with a lot of structure. 
We have primal and dual formulations \citep[Proposition 4.4]{peyre2019computational} given by
\begin{align}
W^{\Cc}_{\varepsilon}(\alpha,\beta) := &  \min_{\pi \in \Pi(\alpha,\beta)} \langle \pi, \Cc \rangle  + \varepsilon KL\left( \pi \| \alpha \otimes \beta \right) \ , 
\label{eq:primal_EOT} \\
\label{eq:dual_EOT}
= & \sup_{f \in \R^n, g \in \R^m} \langle f, \alpha \rangle + \langle g, \beta \rangle  
  - \varepsilon\int_{\Xc \times \Yc}\left(e^{\frac{f(x)+g(y)-\Cc(x,y)}{\varepsilon}}\alpha(dx)\beta(dy)-1\right) \ .
\end{align}

Notice that the dual problem is invariant under the transformation $(f,g) \leftarrow (f,g) + c \left( \mathds{1},- \mathds{1} \right)$. It is generically the only degeneracy of convexity. For fixed $\varepsilon>0$, these are respectively a strongly convex minimization problem and a strongly concave maximization problem. For future reference, we name the objective function 
\begin{align}
	\label{eq:def_Q}
	   Q_{\varepsilon}(f,g)
	:= &  \langle f, \alpha \rangle + \langle g, \beta \rangle - \varepsilon\int_{\Xc \times \Yc}\left(e^{\frac{f(x)+g(y)-\Cc(x,y)}{\varepsilon}}\alpha(dx)\beta(dy)-1\right) \ .
\end{align}

If a maximizer exists, it is unique by strong concavity modulo shifting the pair $(f,g)$ by constants $\left( \mathds{1},- \mathds{1} \right)$. It is denoted $(f^*,g^*)$. The pair $(f^*,g^*)$ is called the optimal Kantorovich potentials. Upon computing the differential of the objective function $Q_\varepsilon$, the first order conditions for optimality read
\begin{align}
	\label{eq:foc_f}
	f = g^{\Cc, \varepsilon} := & -\varepsilon\log \left(\int_{\Xc }\beta(dy)e^{\frac{g(y)-\Cc(\cdot,y)}{\varepsilon}}\right)\ ,
	\end{align}
\begin{align}
	\label{eq:foc_g}
	g = f^{\Cc, \varepsilon} := & -\varepsilon\log \left(\int_{\Yc }\alpha(dx)e^{\frac{f(x)-
	\Cc(x,\cdot)}{\varepsilon}}\right) \ .
\end{align}

In the same fashion as classical OT 
with generalized convexity notions \citep[Definition 2.33]{villani2021topics}, $g^{\Cc, \varepsilon}$ and $g^{\Cc, \varepsilon}$ are called $(\Cc, \varepsilon)$-concave transforms of $f$ and $g$ respectively. 

In any case, the optimal coupling is recovered through the explicit formula \citep[Proposition 4.3]{peyre2019computational}
\begin{align}
\label{eq:Optimal_P_epsilon}
	\pi_\varepsilon^*
:= & \ \exp\left( \frac{f^* \oplus g^*-\Cc}{\varepsilon} \right) \alpha \otimes \beta 
=    \left(\alpha_{i} \beta_{j} \exp\left( \frac{f^*_i+g^*_j-C_{i,j}}{\varepsilon}\right) \right)_{
	  \substack{ 1 \leq i \leq n\\
			\ 1 \leq j \leq m } } \ .
\end{align}

\begin{rmk} 
	Here, both OT and EOT, the primal problem is in the space of dimension $n\times m $ while the dual problem is in the space of dimension $n+m$. Therefore, it is computationally preferable to solve the dual problem instead of the primal.
\end{rmk}

\paragraph{Semi-dual formulation.} Since that, if one of the potentials is fixed, the optimal value for the other is known, the optimization problem of Eq. \eqref{eq:dual_EOT} can be reduced to an optimization over one potential only. This yields the semi-dual formulation, which is formally stated as
\begin{align}
	\label{eq:semi_dual_EOT}
    W^{\Cc}_{\varepsilon}(\alpha,\beta)
	= \sup_{f \in \Lc^1(\alpha)} \langle f, \alpha \rangle +  \langle f^{\Cc, \varepsilon}, \beta \rangle
	= & \sup_{g \in \Lc^1(\beta)} \langle g^{\Cc, \varepsilon}, \alpha \rangle +  \langle g, \beta \rangle \ .
\end{align}
Without loss of generality, we can assume $n<m$ and optimize over $f \in \Lc^1(\alpha)$. For future reference, the objective function is written
\begin{align}
	\label{eq:def_Q_semi}
	   Q_{\varepsilon}^{\mathrm{semi}}(f)
	:= &  \langle f, \alpha \rangle +  \langle f^{\Cc, \varepsilon}, \beta \rangle \ .
\end{align}
For more details, see \citep{cuturi2018semidual}.

\paragraph{The Sinkhorn--Knopp algorithm.}
The idea of the SK algorithm is to successively update the potentials $f$ and $g$, so that each of them satisfies the first order equation i.e.
\begin{align}
	\label{eq:SK_iterations}
	f_{k+1} \longleftarrow & g^{\Cc, \varepsilon}_k \ , 
	&
	g_{k+1} \longleftarrow & f^{\Cc, \varepsilon}_{k+1} \ .
\end{align}
In regards to the block form of the differential, the Sinkhorn algorithm can be seen a block coordinate ascent algorithm \citep[Remark 4.3]{peyre2019computational}. Indeed, by fixing one of the potentials, the other is updated to its optimal value. Then this is repeated alternatively.

In the discrete setting, this only amounts to rescaling columns and rows of the matrix $\left( \exp\left( \frac{f_i+g_j-\Cc_{i,j}}{\varepsilon} \right) \right)_{i,j}$, so that each rescaling enforces the marginal constrain of $\alpha$ or $\beta$, hence the name of iterative proportional fitting. 

% ---------------------------------------------------------
\section{Spectral insight into the ill-conditioning of EOT}
\label{section:spectral_insight}

Let us describe a particular integral operator whose spectral analysis plays a key role in the stability of EOT. For simpler notation, define $H := \Lc^2( \Xc; \alpha) \oplus \Lc^2(\Yc; \beta)$ which is the natural space where the potentials $f \oplus g$ live. Unless otherwise stated, the default norm is the one induced by the scalar product. Define the operator $\mathbb{K}_\varepsilon: H \rightarrow H$ as
\begin{align}
    \label{eq:def_operator_K}
    \mathbb{K}_\varepsilon( f_1 \oplus g_1 )
    := & \ 
    \left[ \Op(\pi_\varepsilon)(g_1) \right] \oplus
    \left[ 
    \Op(\pi_\varepsilon)^T(f_1) \right] \ ,
\end{align}
where  
\begin{align}
    \label{eq:def_Op_pi}
    \Op( \pi_\varepsilon^* ) ( g_1 )
    := & \ 
    \int \exp\left( \frac{f^*(\cdot) + g^*(y) - \Cc(\cdot, y)}{\varepsilon} \right) g_1(y) \beta(dy) \ ,
\end{align}

\begin{align}
    \label{eq:def_Op_pi_T}
    \Op( \pi_\varepsilon^* )^T ( f_1 )
    := & \ 
    \int \exp\left( \frac{f^*(x) + g^*(\cdot) - \Cc(x, \cdot)}{\varepsilon} \right) f_1(x) \alpha(dx) \ .
\end{align}
The notation $\Op( \pi_\varepsilon^* )^T$ indicates the adjoint in the sense that $\langle \Op( \pi_\varepsilon^* ) ( g_1 ), f_1 \rangle_{L^2(\Xc; \alpha)} = \langle g_1, \Op( \pi_\varepsilon^* )^T(f_1) \rangle_{L^2(\Yc; \beta)} = \int \pi_\varepsilon^*(dx dy) f_1(x) g_1(y)$.

\paragraph{Key insight.} The operator $\K_\varepsilon$ appears at two levels showing how crucial it is for the conditioning of the EOT dual problem \eqref{eq:dual_EOT}. On the one hand, the Hessian is intimately linked to $\K_\varepsilon$ via the following structural result which seems to have gone unnoticed in the literature.

\begin{thm}[Structure Theorem for the Hessian]
\label{thm:hessian_structure}
The Hessian of $Q_\varepsilon$ at the optimal points has the following expression
$$ \nabla^2 Q_\varepsilon(f^*, g^*) = \frac{-1}{\varepsilon}\left( \Id_H + \K_\varepsilon \right) \ .$$
Moreover
\begin{itemize}
    \item The spectrum of $\K_\varepsilon$ belongs to $[-1, 1]$ and is symmetric around zero. There is a spectral gap and the spectrum is of the form $\{ \pm \rho_\ell \ ; \ell=0, 1, 2, \dots \}$ with $\rho_0 = 1 > \rho_1 \geq \rho_2 \geq \dots$.
    \item Smallest eigenvalues and associated eigenvectors can be recovered from the largest as $\K_\varepsilon( u  \oplus v ) = \rho (u \oplus v)$ if and only if $\K_\varepsilon( (-u) \oplus v ) = -\rho ( (-u) \oplus v)$.
\end{itemize}
\end{thm}
\begin{proof}
See Supplementary material.
\end{proof}
As a consequence, eigenvalues of $\Id_H+\K_\varepsilon$ near zero have an important influence on all first order (gradient) methods and the stability of a second order algorithm such as damped Newton. 

On the other hand, it also appears as the linearization of the fixed-point SK algorithm as shown in the upcoming Lemma \ref{lemma:equivalent}. As such, the extremal eigenvalues are again important regarding the asymptotic rate of convergence.

\paragraph{Asymptotic rate of convergence.}
In fact, it is easy to describe the asymptotic rate of convergence which is observed in \cite[Fig. 4.5]{peyre2019computational}. The estimate of Eq. \eqref{eq:contraction_coefficient} is therefore sharp, up to the choice of the constant $\kappa$. 
\begin{lemma}
\label{lemma:equivalent}
Let $(f_k, g_k)$ be a sequence converging to $(f^*, g^*)$. We have the following asymptotic equivalent 
\begin{align*}
    & 
    \left(  g_k^{\Cc, \varepsilon} - f^*,
            f_k^{\Cc, \varepsilon} - g^*
    \right) \\
    = & 
    -\K_\varepsilon
    \left(  f_{k} - f^*,
            g_{k} - g^*
    \right) 
    +
    \Oc\left( \frac{1}{\varepsilon}
    \| \left(  f_{k} - f^*,
            g_{k} - g^*
    \right) 
    \|_\infty^2
    \right)
    \ .
\end{align*}
In particular, if $f_k$ is the sequence of SK iterates, then there exist a constant $C>0$ such that
$$ \| f_{k+1} - f^\ast \| \; \le \; \rho_{1}^{\,2} \, \|f_k - f^\ast\| 
   + C \| f_k - f^\ast \|^{2} \ ,$$
\end{lemma}
so that the spectral gap of $\K_\varepsilon$ drives the convergence.
\begin{proof}
See Supplementary material.
\end{proof}

\paragraph{Spectral stability.} A motivated conjecture is as follows. As $\varepsilon \rightarrow 0$, the spectral distribution of $\K_\varepsilon$ seems to concentrate on the boundary $\partial [-1, 1] = \{-1, 1\}$ causing more and more instability. This is illustrated in Figure~\ref{fig:eigenvalues} where the dataset used is uniform measure on the square and annulus.
However, as shown in Fig. \ref{fig:eigenvector_heatmaps}, associated eigenvectors seem rather indifferent to $\varepsilon$'s variation. 

\begin{wrapfigure}{r}{0.3\textwidth}
  \vspace{-0.7cm}
  \begin{center}
    \includegraphics[width=0.3\textwidth]{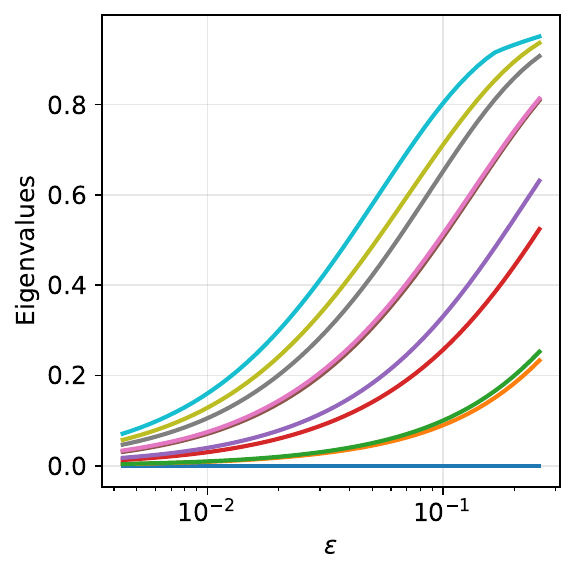}
  \end{center}
  \caption{Behavior of the smallest ten eigenvalues of $\Id_H + \mathbb{K}_\varepsilon$ in function of $\varepsilon$.}
  \label{fig:eigenvalues}
\end{wrapfigure}
We surmise that a proof of this conjecture would involve subtle spectral analysis, and we consider that beyond the scope of the current paper. Nevertheless, a theoretical justification would be as follows. The Structure Theorem \ref{thm:hessian_structure} shows that the Hessian as proportional to $\Id_H + \K_\varepsilon$, and $\K_\varepsilon$ itself can be reinterpreted as the transition matrix of Markov chain. This Markov chain is 2-periodic on a state space with types of points: the points in $\Xc$ and points in $\Yc$. The transitions are given by Gibbs measures with temperature $\varepsilon$. In that context, following \cite{holley1988simulated, miclo1995comportement}, the spectrum is expected to be unstable as the spectral gap disappears, while the eigenvectors are expected to be stable. This also explains why Fig. \ref{fig:eigenvector_heatmaps} is reminiscent of Courant's nodal domains and spectral clustering \cite{chung1997spectral}, by interpreting $\K_\varepsilon$ as a Laplacian or as a graph Laplacian.

\begin{figure*}[t]
  \centering
  \includegraphics[width=1.0\textwidth]{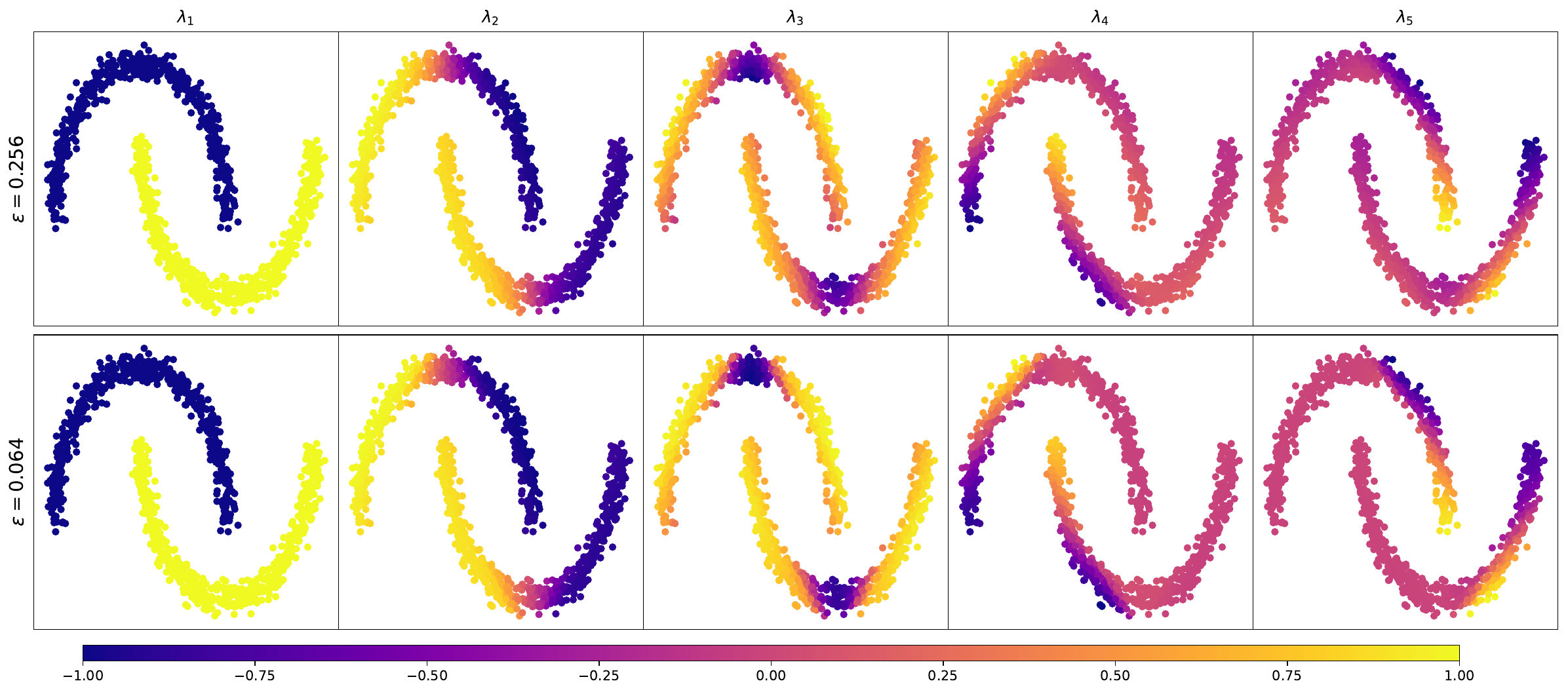}\\
  \includegraphics[width=1.0\textwidth]{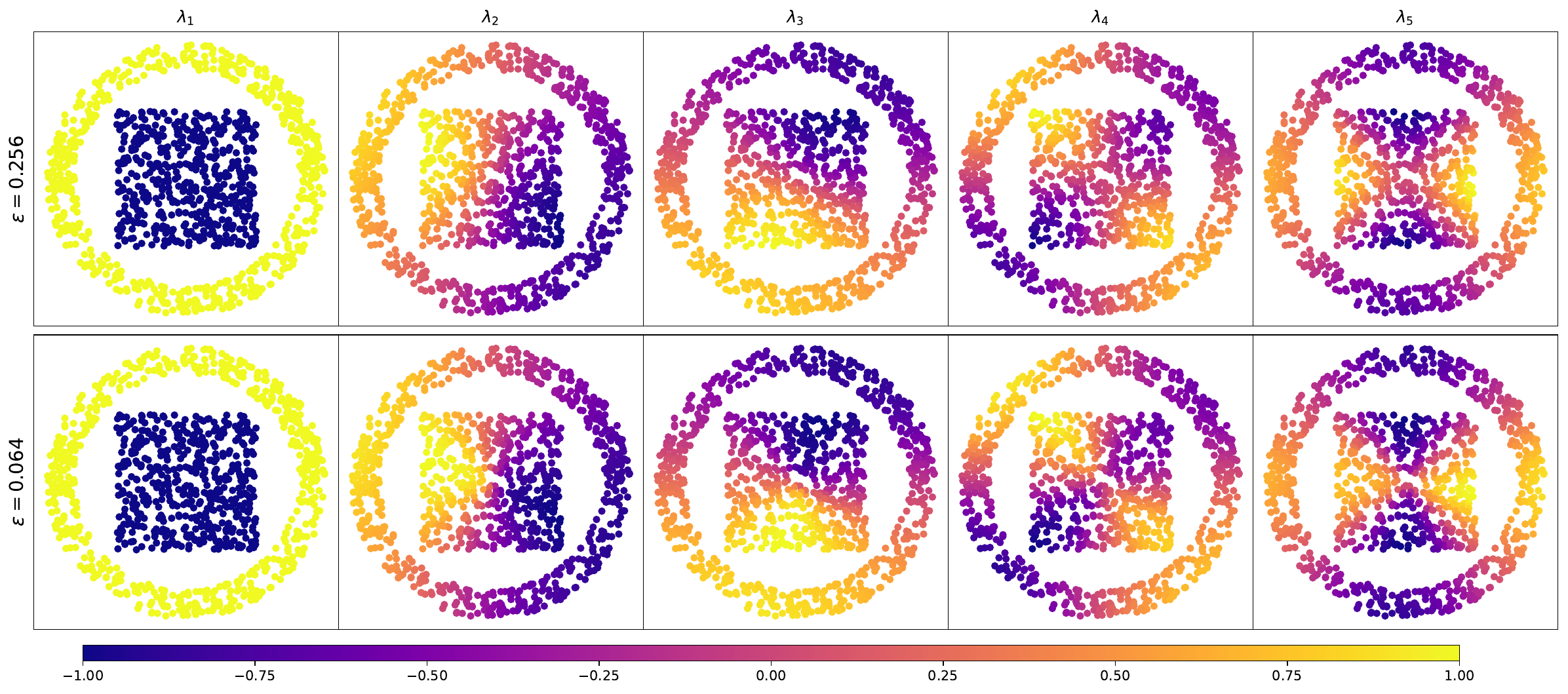}
  \caption{%
    Stability of eigenvectors. Eigenvectors are viewed as functions on the spaces $\Xc$ and $\Yc$, and viewed as heatmaps. One sees how little the 5 first eigenvalues change from $\varepsilon=0.
    256$ to $\varepsilon=0.064$.%
    \textbf{ Top:} Measures are uniform on each component of the classical two-moons dataset.%
    \textbf{ Bottom:} Measures are uniform on the square and an annulus.%
    }
  \label{fig:eigenvector_heatmaps}
\end{figure*}

\paragraph{Relationship to classical acceleration and extrapolation.} We mentioned the paper \citep{a14050143} in introduction. Recall that standard acceleration methods such as Richardson's or Anderson's work as follows. They proceed with averaging the trajectory of the algorithm, in order to produce a better asymptotic behavior. They are generic and require little input beyond the existence of a linearization as in Lemma \ref{lemma:equivalent}. The exact form of the linearization operators is not even required.

In a sense, our approach is a Richardson-type extrapolation. Only that we actually have identified the limiting linearization, and we have an insight on its spectral behavior. It is this information that is injected and leveraged our upcoming SK--NR Algorithm \ref{algorithm:ourMethod}.

% --------------------------------------------------
\section{Algorithm}
\label{section:algorithm}

Our method is described as Algorithm \ref{algorithm:ourMethod}.
This is a variation around Sinkhorn--Knopp algorithm in the ``log-domain''.
Given $\varepsilon>0$, it is run with $V_\ell = V_\ell(\varepsilon')$ being described by a matrix in $\R^{n \times \ell}$ with orthonormal column vectors. These vectors obtained from the spectral decomposition of the Hessian of $Q_{\varepsilon'}$ of $\varepsilon'>\varepsilon$. The initial Hessian of $Q_{\varepsilon'}$ is obtained via a vanilla SK, which is very fast for larger $\varepsilon'$. In that sense, our method is a \emph{spectral warm-start} strategy.

A crucial remark is that a full spectral decomposition is not needed. From Theorem \ref{thm:hessian_structure}, the smallest eigenvalues and their associated eigenvectors can be deduced by symmetry from the largest ones. And $\ell$ largest eigenvalues can be conveniently computed thanks to a partial eigen decomposition, which is based on an efficient power method. Also, the knowledge of the eigen decompositions of $\nabla^2 Q_{\varepsilon'}$ and $\nabla^2 Q_{\varepsilon'}^{\mathrm{semi}}$ are essentially equivalent.

\begin{algorithm}
\caption{SK-NR($\ell)$ Algorithm}
\begin{algorithmic}[1]
\Require Cost matrix $\Cc \in \mathbb{R}^{n \times m}$, regularization parameter $\varepsilon > 0$, target marginals $\alpha, \beta \in \mathbb{R}^n_+$, tolerance $\omega > 0$, max iterations $\mathrm{max\_iter}$, Vector space $V_\ell \subset \R^n$ of dimension $\ell$ of low-frequency modes.
\Ensure Approximate transport plan $\pi$ following the EOT setup
\State Initialize potentials: $f_0 \gets \mathbf{0}_n$, $g_0 \gets \mathbf{0}_n$
\For{$k = 1$ to $\mathrm{max\_iter}$}
    \State $g_k \gets f_{k-1}^{\Cc, \varepsilon}$ \Comment{SK iteration Eq. \eqref{eq:SK_iterations}}
    \State $f_k \gets g_{k}^{\Cc, \varepsilon}$
    \State $\pi \gets \left( \exp\left( \frac{f_i+g_j-\Cc_{i,j}}{\varepsilon} \right) \alpha_i \beta_j \right)_{i,j}$ \Comment{Compute approx. coupling Eq. \eqref{eq:Optimal_P_epsilon}}
    \If{converged (e.g., $\| \pi \mathds{1}_m - \alpha\| + \|\pi^\top \mathds{1}_n - \beta\| < \omega$)}
        \State \textbf{break}
    \EndIf
    \If{ $\dim V_\ell > 0$ } \Comment{Attempts Newton step if provided with low frequency directions}
        \State Compute value function $Q_\varepsilon^{\mathrm{semi}}(f_k)$
        \State Compute 1\textsuperscript{st} and 2\textsuperscript{nd} derivatives on $V_\ell$: 
        $ \nabla_{V_\ell} Q_\varepsilon^{\mathrm{semi}}(f_k), 
          \nabla_{V_\ell}^2 Q_\varepsilon^{\mathrm{semi}}(f_k)
        $
        \State Compute NR update: $\Delta f_k \gets -\left( \nabla_{V_\ell}^2 Q_\varepsilon^{\mathrm{semi}}(f_k) \right)^{-1} \nabla_{V_\ell} Q_\varepsilon^{\mathrm{semi}}(f_k)$
        \State $f_k' \gets f_k + \Delta f_k$
        \If{$Q_\varepsilon^{\mathrm{semi}}(f_k') > Q_\varepsilon^{\mathrm{semi}}(f_k)$}
            \State $f_k \gets f_k'$ \Comment{Accept Newton step}
        \EndIf
    \EndIf
\EndFor
\end{algorithmic}
\label{algorithm:ourMethod}
\end{algorithm}
\addtocounter{thm}{1}

\paragraph{Per iteration complexity.}
The complexity of a Sinkhorn--Knopp iterations is $\Oc\left( n m \right)$ as it dominated by the cost of the matrix-vector multiplications $K e^{\frac{f}{\varepsilon}}$, $K^T e^{\frac{g}{\varepsilon}}$. Notice that log-domain regularization incurs the additional cost of exponentiating multiple times by forming matrices such as $\left( e^{\frac{f_i+g_j-C_{i,j}}{\varepsilon}} \right) \in \R_+^{n \times m}$, as well as resizing by row minima or column minima. This remains of order $\Oc( mn )$ though.

The cost of a Newton step includes
\begin{itemize}
    \item forming gradient and Hessian restricted to the smaller dimensional space $V_\ell$. This has cost $\Oc\left( \ell (m+n) + mn \right)$.
    \item inverting linear system. This has cost $\Oc\left( \ell^3 \right)$. Notice that one can invoke the Cholesky decomposition for slightly more efficiency, as the Hessian is symmetric and negative-definite.
    \item evaluating the objective function. This has cost $\Oc( n m )$.
\end{itemize}

All in all, the complexity of one iteration for our method is $\Oc\left( \ell^3 + nm \right)$. For $\ell=0$, we morally recover the usual Sinkhorn--Knopp iteration. Both complexities remain very close even for large $\ell \ll \left( nm \right)^{\frac13}$. In practice, when comparing SKNR to the usual SK in our experiments, the overhead per iteration proved to be minimal in terms of compute time.

\paragraph{Convergence.} Recall that convergence of the SK algorithm is guaranteed via a contraction argument using the Hilbert projective distance -- see \cite[Theorem 4.2]{peyre2019computational}. Conversely, the NR step is designed to increase the value function $Q_\varepsilon^{\mathrm{semi}}$, offering rapid local improvement when applicable. However, when combined in an alternating scheme, the contraction property of SK and the ascent guarantee of NR do not straightforwardly imply global convergence, as each step may inadvertently disrupt the structure or progress established by the other. Therefore, despite the intuitive appeal of this hybrid approach, convergence is not immediate and must be rigorously established through a dedicated proof.

\begin{thm}
\label{thm:convergence}
Algorithm \ref{algorithm:ourMethod} converges to the EOT optimum $(f^*, g^*, \pi^*)$. Moreover, if $V_\ell = V_\ell(\varepsilon)$ is the eigenspace corresponding the first $\ell$-eigenvectors of $\nabla^2 Q_\varepsilon^{\mathrm{semi}}$, we have the error estimate
\begin{align}
\label{eq:SKNR_error_estimate}
       \left\| f_{k+1} - f^* \right\| 
\leq & \ \rho_{\ell+1}^2 \left\| f_{k}-f^* \right\| + C \left\| f_{k} - f^* \right\|^2 \ ,
\end{align}
where $\rho_\ell$ is the $\ell$-th eigenvalue value of $\K_\varepsilon$, and $C$ is an implicit constant .
\end{thm}
\begin{proof}
See Supplementary material. We give a sketch of proof here.
We first show that the semi-dual $Q_\varepsilon^{\mathrm{semi}}$ increases, that is $Q_\varepsilon^{\mathrm{semi}}(f_k) \geq Q_\varepsilon^{\mathrm{semi}}(f_{k-1})$.
Then, we obtain the convergence in spite of the NR step. We prove that each step of a SK iteration guarantees sufficient increase. This uses both strong concavity for the oscillation norm and the global contraction for the Hilbert distance.
\end{proof}
By inspecting the proof, ones realizes that convergence is still ensured upon replacing the NR step by \emph{any} optimization step which increases the value function $Q_\varepsilon^{\mathrm{semi}}$. This is only possible thanks to the concavity properties of the semi-dual formulation of EOT.
It turns out that the NR step is a simple, yet very efficient way to increase this value as shown empirically in Section~\ref{section:experiments}. The spectral insight allows to maximize the effect of the NR step, while remaining low-dimensional.

\paragraph{Global complexity.}
Let $\omega\in(0,1)$ be the stopping tolerance on the potentials.
Because the SK iteration contracts linearly with factor $\rho_1(\mathbb{K}_\varepsilon)^2$, while {SK–NR}$(\ell)$ contracts with $\rho_\ell(\mathbb{K}_\varepsilon)^2$ thanks to Theorem~\ref{thm:convergence}, the iteration counts satisfy
$$
  N_{\mathrm{SK}}(\omega) \simeq \Bigl\lceil\frac{\log \omega}{2\log \rho_1(\mathbb{K}_\varepsilon)}\Bigr\rceil,
  \qquad
  N_{\mathrm{SK-NR}(\ell)}(\omega) \simeq \Bigl\lceil\frac{\log \omega}{2\log \rho_\ell(\mathbb{K}_\varepsilon)}\Bigr\rceil.
$$
\noindent\emph{Asymptotics for small $\varepsilon$.}
The heuristic from small-temperature Schr\"odinger operators \cite{miclo1995comportement} and from \citep[Remark 4.16]{peyre2019computational} is that, as $\varepsilon\!\to\!0$,
$$
  \rho_\ell(\mathbb{K}_\varepsilon) \simeq 1-\exp \Bigl(-\frac{\lambda_{\ell+1}}{\varepsilon}\Bigr),
$$
for certain decreasing $\lambda_j$'s.
Injecting this estimate yields the heuristic, yet informative bounds
$$
  N_{\mathrm{SK}}(\omega) \asymp 
  - \exp \Bigl(\frac{\lambda_{1}}{\varepsilon}\Bigr) \log\omega,
  \qquad
  N_{\mathrm{SK-NR}}(\omega) \asymp
  -\exp \Bigl(\frac{\lambda_{\ell+1}}{\varepsilon}\Bigr) \log\omega.
$$
Hence each additional Newton direction replaces the dominant exponent $\lambda_1$ by $\lambda_{\ell+1}$, producing an \emph{exponential} drop in the total iteration count when $\varepsilon$ is small, consistent with the speed-ups observed in Section~\ref{section:experiments}. Empirically, Fig. \ref{fig:eigenvalues} shows how changing $\ell$ changes the contraction speed, without relying on the previous heuristic.

\paragraph{Remarks on design choices of the algorithm.} Lines 14-16 are there to make sure that the Newton step does not decrease the objective function, which can indeed happen since Newton's convergence is only local. In fact, as $\varepsilon \rightarrow 0$, removing those lines leads to quite a bit of instability.

In order to fully leverage a Newton step, one would argue for adding a line search component. However, we found that this only complicates the algorithm while an SK step is not less efficient that a damped Newton step in that range -- away from the basin of attraction. A line search procedure also renders the algorithm less parallelizable.

\paragraph{Disadvantages of the algorithm compared to vanilla SK.} Two shortcomings come to mind. On the one hand, vanilla SK can be automatically differentiated as is, whereas SK-NR($\ell$) loses that property because of the NR step. This can be easily fixed though thanks to implicit differentiation~\citep{luise2018differential}, rather than back-propagating through the loop~\citep{genevay2018learning,pauwels2023derivatives}.

On the other hand, by design, vanilla SK is fully parallel and GPU friendly allowing it to solve several EOT problems simultaneously -- see \citep[Remark 4.16]{peyre2019computational}. Indeed, proportionnal fitting in parallel only amounts to multiplying higher order tensors which is not the case for the NR step. The NR step can be made truly parallel and GPU friendly by pursuing iterative inversion of $\nabla^2_{V_\ell} Q_\varepsilon^{\mathrm{semi}}$ or using GPU friendly implementations of conjugate gradient.

% --------------------------------------------------
\section{Numerical experiments}
\label{section:experiments}

We evaluate the performance of SK-NR on two different tasks.
First, we show the behaviour on synthetic 2D distributions to highlights the impact of the different parameters of the method.
Second, we show that the method interesting improvements for a single-cell gene expression task when decreasing the value of $\varepsilon$ and using a small number of eigenvalues of the matrix $\mathbb{K}_\varepsilon$ (here $\ell=30$).

\paragraph{Estimating couplings for 2D distributions.}
\begin{figure*}[t]
  \centering
  \includegraphics[width=0.32\textwidth]{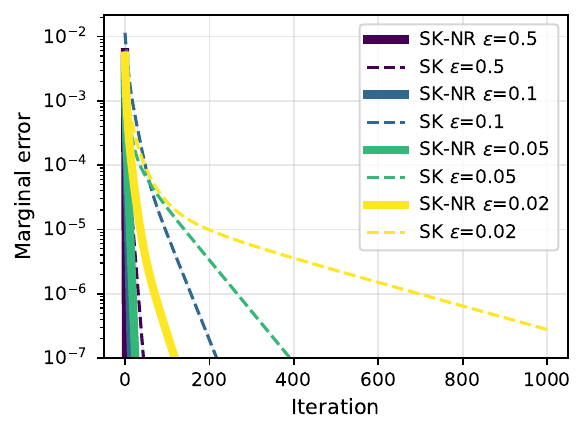}
  \hfill
  \includegraphics[width=0.32\textwidth]{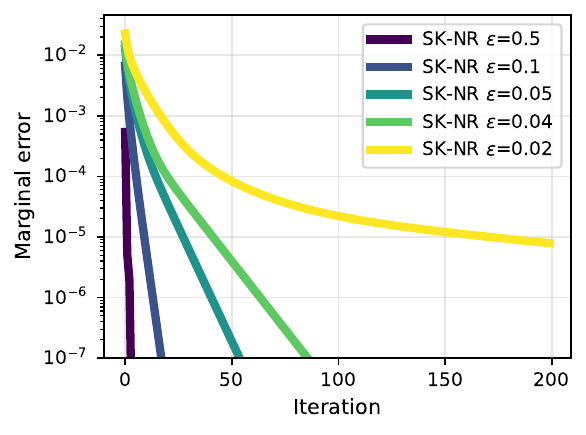}
  \hfill
  \includegraphics[width=0.32\textwidth]{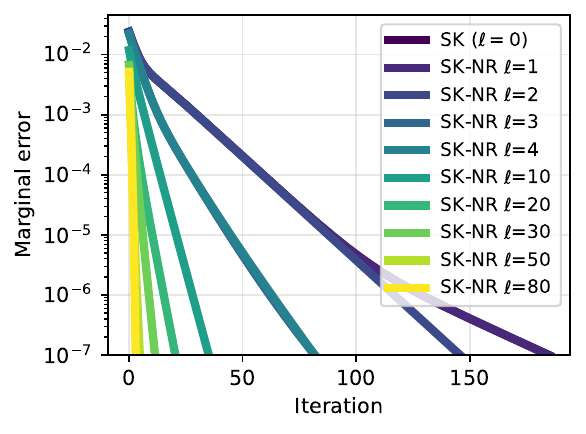}
  \caption{%
    All plots show the influence of the number of iterations with respect to the marginal error $\|\pi^\top f - \beta\|$.%
    \textbf{ Left:} Influence of sequential warm-starting by initializing with $\varepsilon=1$ and progressively decreasing it. We compare the standard Sinkhorn warm-start (dashed) with SK-NR using spectral information.%
    \textbf{ Middle:} Influence of a static warm-start with $\varepsilon=1$ when transitioning to a smaller value of $\varepsilon$.%
    \textbf{ Right:} Influence of the number of eigenvalues used, $\ell=0$ corresponds to the standard Sinkhorn algorithm.
    }
  \label{fig:2d-gaussians}
\end{figure*}
We consider a synthetic optimal transport problem between two Gaussian point clouds in $\mathbb{R}^2$.
The source distribution consists of $n = 200$ samples drawn from a centered isotropic Gaussian with mean $\mu_s = (0, 0)$ and covariance $\Sigma_s = I$.
The target distribution is defined by $2n = 400$ samples drawn from a correlated Gaussian with mean $\mu_t = (4, 4)$ and covariance $\Sigma_t = \begin{bmatrix}1 & -0.8 \\ -0.8 & 1\end{bmatrix}$.
Both source and target distributions are uniform over their respective samples.
The cost matrix $M$ is defined as the squared Euclidean distance between all source and target samples. Figure~\ref{fig:2d-gaussians} presents a comparison of Sinkhorn-based methods in this setting, showing the marginal error $\|\pi^\top f -\beta\|$ as a function of the number of iterations under various warm-starting and spectral regularization strategies.
In particular, we show the tremendous improvement when used as sequential warmstart towards a target $\varepsilon$ value with respect to a standard Sinkhorn warmstart of the dual potentials.
Note however that it is not a method to achieve lower values of $\epsilon$ due to potential instabilities in the computation of the eigenvectors in this case.

\paragraph{Single-cell gene expression.}
We consider a real-world application of map estimation that arises in studies like~\citep{schiebinger2019optimal} which aim to infer cellular evolution from temporal gene expression. The dataset consists of cell measurements over 18 days, with each sample containing $\approx$1000 gene expression values. As advocated, we reduce dimensionality to $\mathbb{R}^{30}$ using PCA and normalize the data. The source and target distributions, denoted $X$ and $Y$, correspond to samples from day 7 and day 8, respectively.
The cost matrix is of size $6507 \times 3815$ between day 7 and 8, and of size $3815 \times 2982$ between day 8 and 9.
We report in Figure~\ref{fig:wot} the performance of the standard Sinkhorn (dashed) with SK-NR using spectral information. We used $\ell=30$ eigenvectors for all experiments.
\begin{figure}[t]%{r}{0.5\textwidth}
  \begin{center}
    \includegraphics[width=0.45\textwidth]{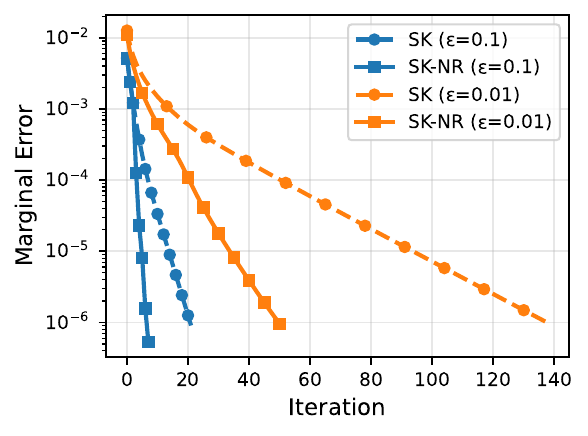} \hfill
    \includegraphics[width=0.45\textwidth]{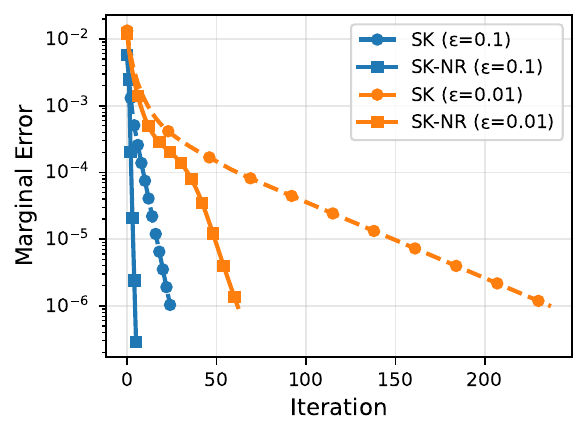}
  \end{center}
  \caption{SK-NR (plain) vs vanilla Sinkhorn (dashed) using decreasing value of $\varepsilon$ on WOT dataset. The first run $\varepsilon = 0.5$ is warm-started by a vanilla SK at $\varepsilon = 1.0$. \textbf{(Left)} Optimal coupling between day 7 and 8. \textbf{(Right)} Between day 8 and 9.}
  \label{fig:wot}
\end{figure}

% --------------------------------------------------
\section{Conclusion}
\label{section:conclusion}
In this paper, we presented a method for accelerating the ubiquitous SK algorithm by injecting spectral information in the process. This leads to an algorithm which we named SK--NR($\ell$), where $\ell$ is a hyperparameter that chooses the number of eigenvectors used for acceleration.
Acceleration is dramatic for smaller $\varepsilon$, which reflects how poorly the SK algorithm behaves as $\varepsilon \rightarrow 0$.

For future research, we are convinced that proving the stability of the spectral projectors associated to $\mathbb{K}_\varepsilon$, is a worthwhile and challenging mathematical endeavor. This paper illustrates the practical implications of this stability property. Finally, the new research directions opened by this paper make it timely to revisit existing annealed and scheduling strategies when taking $\varepsilon \rightarrow 0$ in the line of~\citep{chizat2024annealedsinkhornoptimaltransport}.

%% file: appendix.tex
\appendix

\section{Supplementary Material}

\subsection{Preliminary formulas}
Let us start by recording formulas which will be useful in the proofs. Recall the following. For $(f,g) \in L^2( \Xc ; \alpha) \times L^2( \Yc ; \beta )$, we write $f \oplus g: (x,y) \mapsto f(x)+g(y)$. $H = L^2( \Xc ; \alpha) \oplus L^2( \Yc ; \beta )$
is a Hilbert space endowed with scalar product $\langle f_1 \oplus g_1, f_2 \oplus g_2 \rangle_H = \langle f_1, g_1 \rangle_{L^2( \Xc ; \alpha)} + \langle f_2, g_2 \rangle_{L^2( \Yc ; \beta)}$.

\paragraph{Oscillation norm.}
Define the oscillation semi-norm on both spaces $L^2( \Xc ; \alpha)$ and $L^2( \Yc ; \beta )$ as
$$
    \| \cdot \|_{\mathrm{osc}}: f \mapsto \| f \|_{\mathrm{osc}} = \max f - \min f \ .
$$
Having vanishing semi-norm is equivalent to being constant. It is thus a norm on functions modulo $\R \1$.

In order to metrize $L^2( \Xc ; \alpha) \times L^2( \Yc ; \beta )$ modulo shifts by $(\1, -\1)$, we shall start from any norm $\| \cdot \|$ on functions of two variables and then, on $H$ use
$$
    (f,g) \mapsto \| f \oplus g \| \ .
$$
The default norm is induced from the scalar product on $H$.

\paragraph{For the dual formulation.}
Recall that for $(f,g) \in L^2( \Xc ; \alpha) \times L^2( \Yc ; \beta )$, the objective function from Eq. \eqref{eq:def_Q} is
\begin{align}
\label{eq:def_Q_appendix}
    Q_\varepsilon(f,g)
:= & \langle f, \alpha \rangle
  + \langle g, \beta \rangle
  - \varepsilon \left( -1 + \iint \exp\left( \frac{f \oplus g - \Cc}{\varepsilon} \right) \alpha \otimes \beta \right) .
\end{align}
The following expressions for the first and second order differentials are straightforwardly obtained by Taylor-expansion. The differential along the direction $(f_1, g_1)$ at the point $(f,g)$ is
\begin{align}
\label{eq:dQ}
    dQ_\varepsilon(f,g) \cdot (f_1, g_1)
:= & \langle f_1, \alpha \rangle
  + \langle g_1, \beta \rangle
  - \iint \left( f_1 \oplus g_1 \right) \exp\left( \frac{f \oplus g - \Cc}{\varepsilon} \right) \alpha \otimes \beta .
\end{align}
Upon setting this linear form to zero, we find the first order optimality condition. We recognize the Sinkhorn equations for the pair $(f, g)$ 
\begin{align}
\label{eq:FOC_SK}
    0 & = 1 - \exp\left( \frac{f(\cdot)}{\varepsilon} \right) \int \exp\left( \frac{g(y) - \Cc(\cdot, y)}{\varepsilon} \right) \beta(dy) \ , \\
   0 & = 1 - \exp\left( \frac{g(\cdot)}{\varepsilon} \right) \int \exp\left( \frac{f(x) - \Cc(x, \cdot)}{\varepsilon} \right) \alpha(dx) \ .
\end{align}
The second order differential is the quadratic form
\begin{align}
\label{eq:d2Q}
    d^2 Q_\varepsilon(f,g) \cdot (f_1, g_1) \cdot (f_1, g_1)
:= &
  - \frac{1}{\varepsilon} \iint \left( f_1 \oplus g_1 \right)^2 \exp\left( \frac{f \oplus g - \Cc}{\varepsilon} \right) \alpha \otimes \beta .
\end{align}
Noticing that this operator is non-positive, we deduce that $Q_\varepsilon$ is globally concave. Furthermore, the second order differential is negative definite for $(f_1, g_1)$ considered modulo $\R (\1 , -\1)$, hence strong concavity in the following sense. For every compact set $K \subset H$, there exist $\mu_K$ such that $(f \oplus g, f' \oplus g')  \in K \times K$
\begin{align}
\label{eq:Q_strong_concavity}
    Q_\varepsilon(f',g')
\leq & Q_\varepsilon(f',g')
  + dQ_\varepsilon(f,g) \cdot (f'-f, g'-g)
  - \frac{\mu_K}{2} \| (f'-f) \oplus (g'-g) \|^2 .
\end{align}
\paragraph{For the semi-dual formulation.} Recall that the semi-dual formulation uses an objective function $Q_\varepsilon^{\mathrm{semi}}$ from Eq. \eqref{eq:def_Q_semi}. It is given for $f \in L^2( \Xc ; \alpha)$ by
\begin{align}
\label{eq:def_Q_semi_appendix}
    Q_\varepsilon^{\mathrm{semi}}(f) = Q_\varepsilon( f, f^{\Cc, \varepsilon})
:= & \langle f, \alpha \rangle
  + \langle f^{\Cc, \varepsilon}, \beta \rangle .
\end{align}
In fact, this function is defined on $L^2( \Xc ; \alpha) \mod \R \1$ to account for shifts by constants.
In order to compute differentials, we need to perturb the Sinkhorn transform $f \mapsto f^{\Cc, \varepsilon}$. For small $\eta>0$, and fixed function $f_1$, we have
\begin{align}
  & \left( f + \eta f_1 \right)^{\Cc, \varepsilon}(y) - f^{\Cc, \varepsilon}(y) \\
= & - \eta \int f_1(x) \exp\left( \frac{f(x)+f^{\Cc, \varepsilon}(y)-\Cc(x,y)}{\varepsilon} \right) \alpha(dx) \nonumber\\
  & -\frac{\eta^2}{2 \varepsilon}
  \Bigl[ \int f_1(x)^2 \exp\left( \frac{f(x)+f^{\Cc, \varepsilon}(y)-\Cc(x,y)}{\varepsilon} \right) \alpha(dx) \nonumber \\
  & - \left( \int f_1(x) 
       \exp\left( \frac{f(x)+f^{\Cc,\varepsilon}(y)-\Cc(x,y)}{\varepsilon} 
       \right) \alpha(dx) \right)^2
  \Bigr]
  + \Oc( \eta^3 ) \nonumber \ .
\end{align}
From this, one easily deduces the expressions for the first and second order differentials. 
\begin{align}
\label{eq:dQ_semi}
    dQ_\varepsilon^{\mathrm{semi}}(f,g) \cdot f_1
:= & \ \langle f_1, \alpha \rangle
  - \iint f_1(x) \exp\left( \frac{f(x) + f^{\Cc,\varepsilon}(y) - \Cc(x,y)}{\varepsilon} \right) (\alpha \otimes \beta) (dxdy) \ ,
\end{align}
\begin{align}
\label{eq:d2Q_semi}
    d^2 Q_\varepsilon^{\mathrm{semi}}(f) \cdot f_1 \cdot f_1
:= &
  - \frac{1}{\varepsilon} \Big[ \iint f_1(x)^2 \exp\left( \frac{f(x) + f^{\Cc,\varepsilon}(y) - \Cc(x,y)}{\varepsilon} \right) (\alpha \otimes \beta) (dxdy) \\
   & 
   - \int \beta(dy) \left( \int f_1(x) 
       \exp\left( \frac{f(x)+f^{\Cc,\varepsilon}(y)-\Cc(x,y)}{\varepsilon} 
       \right) \alpha(dx) \right)^2
   \Big] \nonumber \ .
\end{align}

At this level, notice that 
$$
\int f_1(x)^2 \exp\left( \frac{f(x)+f^{\Cc, \varepsilon}(y)-\Cc(x,y)}{\varepsilon} \right) \alpha(dx) 
- 
\left( \int f_1(x) 
       \exp\left( \frac{f(x)+f^{\Cc,\varepsilon}(y)-\Cc(x,y)}{\varepsilon}
       \right) \alpha(dx) \right)^2
$$
is a variance which we write $\mathrm{Var}_{\Cc,\varepsilon, f, y}( f_1 )$. The subscript indicates that we use a measure which depends on all these parameters. This variance is controlled from above and below by the oscillation norm.
Strong concavity ensues, in the following sense.
For every compact set $K \subset L^2(\Xc ; \alpha) \mod \R \1$, there exist $\mu_K$ such that $(f , f')  \in K \times K$
\begin{align}
\label{eq:Q_semi_strong_concavity}
    Q_\varepsilon^{\mathrm{semi}}(f')
\leq & Q_\varepsilon(f')
  + dQ_\varepsilon(f) \cdot (f'-f)
  - \frac{\mu_K}{2} \| f'-f \|_{\textrm{osc}}^2 .
\end{align}

\subsection{Proof of Theorem \ref{thm:hessian_structure}}

\paragraph{Step 1: Expression for the Hessian.}
We consider the second order differential from Eq. \eqref{eq:d2Q} and perform simplifications thanks to the Sinkhorn Equations \eqref{eq:FOC_SK}
\begin{align*}
  & - \varepsilon \ d^2 Q_\varepsilon(f^*,g^*) \cdot (f_1, g_1) \cdot (f_1, g_1) \\
= &
  \iint \left( f_1 \oplus g_1 \right)^2(x,y) \exp\left( \frac{(f^* \oplus g^*)(x,y) - \Cc(x,y)}{\varepsilon} \right) (\alpha \otimes \beta)(dxdy) \\
= & \int \alpha(dx) f_1(x)^2 \int \beta(dy) \exp\left( \frac{f^*(x) + g^*(y) - \Cc(x,y)}{\varepsilon} \right)\\
  & + \int \beta(dy) g_1(y)^2 \int \alpha(dx) \exp\left( \frac{f^*(x) + g^*(y) - \Cc(x,y)}{\varepsilon} \right)\\
  & + 2 \iint f_1(x) g_1(y) \exp\left( \frac{f^*(x) + g^*(y) - \Cc(x,y)}{\varepsilon} \right) (\alpha \otimes \beta)(dx dy)\\
= & \int \alpha(dx) f_1(x)^2 + \int \beta(dy) g_1(y)^2\\
  & + 2 \iint f_1(x) g_1(y) \exp\left( \frac{f^*(x) + g^*(y) - \Cc(x,y)}{\varepsilon} \right) (\alpha \otimes \beta)(dx dy) \ .
\end{align*}
Now in order to identify the Hessian operator, the previous expression needs to be written as a scalar product
$$
   \langle -\varepsilon \nabla^2 Q_\varepsilon(f^*, g^*) (f_1, g_1 ), (f_1, g_1) \rangle_H \ .
$$
By identification, we see that
$$
-\varepsilon \nabla^2 Q_\varepsilon(f^*, g^*) (f_1, g_1 )
= (f_1, g_1) + \left( \Op( \pi_\varepsilon^* ) ( g_1 ), \Op( \pi_\varepsilon^* )^T ( f_1 ) \right) \ .
$$
This is indeed the announced form $\Id_H + \K_\varepsilon$.

\paragraph{Step 2: Spectral properties.} 

\paragraph{Step 2.1: Spectrum is real and symmetric around zero.}
In an orthogonal basis adapted to the decomposition $H = L^2( \Xc ; \alpha) \oplus L^2( \Yc ; \beta )$, the operator $\K_\varepsilon$ takes the form
\begin{align}
    \label{eq:K_form}
    \K_\varepsilon
    & = \begin{pmatrix}
    0 & R \\
    R^T &  0
    \end{pmatrix} \ ,
\end{align}
which is a symmetric real operator. Its spectrum is therefore real. Such a matrix is known as the Hermitization of $R$, and its spectrum is provided by $\pm \lambda$ for $\lambda$ a singular value of 
$R$. Indeed 
$$
\begin{pmatrix}
    0 & R \\
    R^T &  0
    \end{pmatrix}
\begin{pmatrix}
    u \\ v
\end{pmatrix}
= \lambda \begin{pmatrix}
    u \\ v
\end{pmatrix}
$$
is equivalent to the following statements. On the one hand, it is equivalent to the pair of equations
\begin{align*}
   R v = \lambda u \ ,
   \quad
   R^T u = \lambda v \ ,
\end{align*}
which implies that
\begin{align*}
R^T R v = \lambda^2 v \ ,
   \quad
   R R^T u = \lambda^2 u \ .
\end{align*}
This proves the relationship to singular values. On the other hand, it is also equivalent to
$$
\begin{pmatrix}
    0 & R \\
    R^T &  0
    \end{pmatrix}
\begin{pmatrix}
    -u \\ v
\end{pmatrix}
= -\lambda \begin{pmatrix}
    -u \\ v
\end{pmatrix} \ ,
$$ 
which proves that $\lambda$ an eigenvalue if and only if $-\lambda$ is an eigenvalue, with a simple correspondence between eigenvectors.

\paragraph{Step 2.2: Spectrum is in $[-1, 1]$.} Going back to Eq. \eqref{eq:K_form}, if the chosen basis is proportional to indicators of points (Dirac basis), the matrix representation has non-negative coefficients -- $R_{i,j}$ being the value of an explicit positive integral. Therefore, the Perron-Frobenius Theorem applies. The matrix is clearly irreducible and $2$-periodic. Furthermore the Perron-Frobenius positive eigenvector is the constant function since $\K_\varepsilon( \1_\Xc  \oplus \1_\Yc ) = \1_\Xc  \oplus \1_\Yc$, by virtue of the Sinkhorn Equations \eqref{eq:FOC_SK}.

By uniqueness of the Perron-Frobenius eigenvector, which is the unique positive eigenvector whose corresponding eigenvalue has maximal modulus, all the spectrum must be included in $[-1, 1]$. Since we also have a spectral gap, positive eigenvalues are $1 = \rho_0 > \rho_1 \geq \dots$. Negative eigenvalues are simply the opposite, by symmetry. This concludes our proof.

\subsection{Proof of Lemma \ref{lemma:equivalent}}
\begin{proof}
We start by the SK equations \eqref{eq:SK_iterations} and the first order condition \eqref{eq:foc_f}
\begin{align*}
      & g_{k}^{\Cc, \varepsilon} - f^* \\
    = & - \varepsilon \log \int \beta(dy) \exp\left( \frac{g_{k}(y) - \Cc(\cdot, y)}{\varepsilon} \right) - f^*\\
    = & - \varepsilon \log \int \beta(dy) \exp\left( \frac{f^* (\cdot) + g^* (y) - \Cc(\cdot, y)}{\varepsilon} \right) \exp\left( \frac{g_{k}(y) - g^* (y)}{\varepsilon} \right)\\
    = & - \varepsilon \log\left[ 1 +  \int \beta(dy) \exp\left( \frac{f^* (\cdot) + g^* (y) - \Cc(\cdot, y)}{\varepsilon} \right) \left[ \exp\left( \frac{g_{k}(y) - g^* (y)}{\varepsilon} \right) - 1 \right] \right] \ .
\end{align*}
Upon performing  Taylor expansion of $\exp$ at $0$, and of $\log$ at $1$, then we have
\begin{align*}
      & g_{k}^{\Cc, \varepsilon} - f^* \\
    = & - \varepsilon \log\left[ 1 +  \int \beta(dy) \exp\left( \frac{f^* (\cdot) + g^* (y) - \Cc(\cdot, y)}{\varepsilon} \right) \left[ \frac{g_{k}(y) - g^* (y)}{\varepsilon} + \Oc\left( \frac{1}{\varepsilon^2}\|g_{k} - g^*\|_\infty^2 \right)  \right] \right] \\
    = & - \varepsilon \log\left[ 1 +  \frac{1}{\varepsilon} \int \beta(dy) \exp\left( \frac{f^* (\cdot) + g^* (y) - \Cc(\cdot, y)}{\varepsilon} \right) \left[ g_{k}(y) - g^* (y) \right] + \Oc\left( \frac{1}{\varepsilon^2}\|g_{k} - g^*\|_\infty^2 \right)   \right] \\
    = & - \int \beta(dy) \exp\left( \frac{f^* (\cdot) + g^* (y) - \Cc(\cdot, y)}{\varepsilon} \right) \left[ g_{k}(y) - g^* (y) \right] + \Oc\left( \frac{1}{\varepsilon}\|g_{k} - g^*\|_\infty^2 \right) \ .
\end{align*}
We recognize the $L^2(\Xc ; \alpha)$ part of $-\K_\varepsilon (f_k-f^*, g_k-g^*)$. The computation of $f_{k}^{\Cc, \varepsilon}-g^*$ is done in entirely the same fashion.
\end{proof}

\subsection{On the relation between Hessians of dual and semi-dual formulation}

We mentioned in Section \ref{section:algorithm} that the dual and semi-dual Hessians $\nabla^2 Q_\varepsilon(f^*, g^*)$ and $\nabla^2 Q_\varepsilon^{\mathrm{semi}}(f^*)$ basically contain the same spectral information. Let us now detail the precise relationship between them. We basically preferred invoking the semi-dual formulation in the algorithm for the sake of computational efficiency. 

At the optimal point $f = f^*$, thanks to the Sinkhorn equations \eqref{eq:FOC_SK}, Eq. \eqref{eq:d2Q_semi} becomes
\begin{align*}
   & d^2 Q_\varepsilon^{\mathrm{semi}}(f^*) \cdot f_1 \cdot f_1 \\
 = &
  - \frac{1}{\varepsilon} \Big[ \iint f_1(x)^2 \exp\left( \frac{f^*(x) + g^*(y) - \Cc(x,y)}{\varepsilon} \right) (\alpha \otimes \beta) (dxdy) \\
   & 
   - \int \alpha(dx) f_1(x) \int \beta(dy) \exp\left( \frac{f^*(x)+g^*(y)-\Cc(x,y)}{\varepsilon} \right)
     \int \alpha(dx') 
       f_1(x') 
       \exp\left( \frac{f^*(x')+g^*(y)-\Cc(x,y)}{\varepsilon}
       \right)
   \Big] \\
 = &
  - \frac{1}{\varepsilon} \Big[ \int f_1(x)^2 \alpha(dx) \\
   & 
   - \int \alpha(dx) f_1(x) \int \beta(dy) \exp\left( \frac{f^*(x)+g^*(y)-\Cc(x,y)}{\varepsilon} \right)
     \mathrm{Op}(\pi_\varepsilon^*)^T(f_1)(y)
   \Big] \\
 = &
  - \frac{1}{\varepsilon} 
  \Big[ \int f_1(x)^2 \alpha(dx)
   - \int \alpha(dx) f_1(x)
     \left[ \mathrm{Op}(\pi_\varepsilon^*)
     \circ \mathrm{Op}(\pi_\varepsilon^*)^T \right] (f_1)(x)
   \Big] \\
 = & \langle f_1, \nabla^2 Q_\varepsilon^{\mathrm{semi}}(f^*) f_1 \rangle_{L^2(\Xc;\alpha)} \ .
\end{align*}
Upon identification, we find that
\begin{align}
    \label{eq:semidual_Hessian}
    \nabla^2 Q_\varepsilon^{\mathrm{semi}}(f^*)
    = -\frac{1}{\varepsilon}
      \left( \Id - \mathrm{Op}(\pi_\varepsilon^*)
     \circ \mathrm{Op}(\pi_\varepsilon^*)^T \right) \ .
\end{align}
In the block notations of Eq. \eqref{eq:K_form}, this is just saying that 
$    \nabla^2 Q_\varepsilon^{\mathrm{semi}}(f^*)
    = -\frac{1}{\varepsilon}
      \left( \Id - R R^T \right) \ .$
Because of the Hermitization procedure used in the proof of Theorem \ref{thm:hessian_structure}, indeed the spectral decomposition of $\K_\varepsilon$ is essentially equivalent to that of $R R^T$.

\subsection{Proof of Theorem \ref{thm:convergence}}
\begin{proof}[Proof of convergence]
Let $(f_k, g_k)$ be the pair of Kantorovich potentials at iteration $k$. Our goal is to prove the convergence towards the unique optimizers
\begin{align}
    \label{eq:potential_cv}
    \lim_{k \rightarrow \infty} (f_k, g_k) = (f^*, g^*) \ .
\end{align}

\paragraph{Step 1: The value function increases.}
Let us start by showing that for $k \geq 1$, 
$$
Q_\varepsilon^{\mathrm{semi}}(f_k) \geq Q_\varepsilon^{\mathrm{semi}}(f_{k-1}) \ .
$$
Clearly the NR step only increases the value function $Q_\varepsilon^{\mathrm{semi}}$. So, we only need to deal with the SK iteration. Recall that the SK algorithm can be reinterpreted as a block gradient ascent for the function $Q_\varepsilon$. As such, the update for $f_k$ in line 4 is the best value for given $g_k$. and the update for $g_k$ in line 3 is the best value for given $f_{k-1}$. As such
\begin{align}
\label{eq:inequalities}
Q_\varepsilon^{\mathrm{semi}}(f_k)
= & \ Q_\varepsilon( f_k, f_k^{\Cc, \varepsilon} ) \nonumber \\
\geq & \ Q_\varepsilon( f_k, g_k ) = Q_\varepsilon( g_k^{\Cc, \varepsilon}, g_k ) \nonumber \\
\geq & \ Q_\varepsilon( f_{k-1}, g_k )
= Q_\varepsilon^{\mathrm{semi}}(f_{k-1}) \ .
\end{align}
In particular $f_k$ remains in a compact set $K$ of $L^2(\Xc ; \alpha) \mod \R \1$.

\paragraph{Step 2: Consequences of strong concavity for the oscillation norm.}
Recall that Kantorovich potentials $(f_k, g_k)$ are only defined modulo $\R (\1, -\1)$. Likewise, $f_k$ is only defined modulo a constant. As such the semi-norm
$$
    \| \cdot \|_{\mathrm{osc}}: f \mapsto \| f \|_{\mathrm{osc}} = \max f - \min f
$$
yields a norm on the quotient space. By definition of strong-concavity
Because of strong-convavity as in Eq. \eqref{eq:Q_semi_strong_concavity}, we have for all $(f, f') \in K \times K$
$$
    Q_\varepsilon^{\mathrm{semi}}(f)
    \leq 
    Q_\varepsilon^{\mathrm{semi}}(f')
    +
    dQ_\varepsilon^{\mathrm{semi}}(f')\cdot (f-f')
    -
    \frac{\mu_K}{2} \left\| f - f' \right\|_{\mathrm{osc}}^2 \ .
$$
Setting $f'=f^*$ for the optimum point, one obtains the classical inequality that allows to control distance to the optimum using the optimality gap
\begin{align}
    \label{eq:value_from_gap_control}
    \frac{\mu_K}{2} \left\| f - f^* \right\|_{\mathrm{osc}}^2
    \leq & \ 
    Q_\varepsilon^{\mathrm{semi}}(f^*)
    -
    Q_\varepsilon^{\mathrm{semi}}(f) \ .
\end{align}
Now, we combine Step 1 with Eq. \eqref{eq:value_from_gap_control} to deduce the following. On the one hand, $f_k$ remains in a compact set $K$. On the other hand, if 
$$
\lim_{k \rightarrow \infty} Q_\varepsilon^{\mathrm{semi}}(f_k)
= Q_\varepsilon^{\mathrm{semi}}(f^*) \ ,
$$
then we are done.

\paragraph{Step 3: Conclusion.}
Let assume that 
\begin{align}
\label{eq:assumption}
\lim_{k \rightarrow \infty} Q_\varepsilon^{\mathrm{semi}}(f_k)
= & \ Q_\varepsilon^\infty < Q_\varepsilon^{\mathrm{semi}}(f^*) \ ,
\end{align}
and find a contradiction. Given the explanation of the previous paragraph, that would conclude the proof.

To do so, we shall prove that one SK iteration guarantees sufficient increase, thanks to strong concavity and independently of the NR step. An SK-NR iteration starts from $f_{k-1}$, then computes the SK iteration which we write as
$$
    f_k = \left( f_{k-1}^{\Cc, \varepsilon} \right)^{\Cc, \varepsilon}
        = \left( g_k^{\Cc, \varepsilon} \right)^{\Cc, \varepsilon} \ , 
$$
then performs the low dimensional NR step, which may or may not be successful. There is no loss of generality in ignoring the NR step, hence our notations. Since the SK iteration is a global contraction for the Hilbert distance \cite[Theorem 4.2]{peyre2019computational}, and because the Hilbert projective distance is nothing but the oscillation semi-norm, after taking logarithms, there is a $0 < \lambda = \lambda(\Cc, \varepsilon) < 1$ such that
$$
\left\| f_k - f^* \right\|_{\mathrm{osc}}
\leq \lambda^2 \left\| f_{k-1} - f^* \right\|_{\mathrm{osc}} \ .
$$
By the reverse triangular inequality, we have
\begin{align*}
       \left\| f_{k} - f_{k-1} \right\|_{\mathrm{osc}}
\geq & \ \left\| f_{k-1} - f^* \right\|_{\mathrm{osc}} - \left\| f_k - f^* \right\|_{\mathrm{osc}} \\
\geq & \ (1-\lambda^2) \left\| f_{k-1} - f^* \right\|_{\mathrm{osc}} \ ,
\end{align*}
and the latter quantity remains bounded from below thanks to our standing assumption \eqref{eq:assumption}.

Now let us examine the increase in the inequalities \eqref{eq:inequalities}. By the strong concavity \eqref{eq:Q_strong_concavity} applied for the points $( f_{k-1}, g_k )$ and $( f_{k}, g_k )$, we have
\begin{align*}
      \frac{\mu_K}{2} \| (f_k- f_{k-1}) \oplus 0 \|^2
\leq & Q_\varepsilon( f_{k-1}, g_k )
     - Q_\varepsilon( f_k, g_k )
     + dQ_\varepsilon(f_{k-1}, g_k) \cdot \left( f_k - f_{k-1}, 0 \right) \ .
\end{align*}
The difference $Q_\varepsilon( f_{k-1}, g_k ) - Q_\varepsilon( f_k, g_k )$ goes to zero as $k \rightarrow \infty$ because of the inequalities \eqref{eq:inequalities} and our standing assumption \eqref{eq:assumption}. Also, because we can use any norm, there exist a constant $c>0$ such that
\begin{align*}
      \liminf_{k \rightarrow \infty} \| (f_k- f_{k-1}) \|_{\mathrm{osc}}^2
\leq & c \liminf_{k \rightarrow \infty} dQ_\varepsilon(f_{k-1}, g_k) \cdot \left( f_k - f_{k-1}, 0 \right) \ .
\end{align*}

Because $\| f_k- f_{k-1}\|_{\mathrm{osc}}$ is bounded from below, $dQ_\varepsilon(f_{k-1}, g_k) \cdot \left( f_k - f_{k-1}, 0 \right)$ is eventually also positive bounded away from zero. Therefore there is a gradient step starting from $(f_{k-1}, g_k)$, only in the direction of $L^2(\Xc;\alpha)$, which yields a sufficient increase.

This yields a contradiction. By the block ascent interpretation of the Sinkhorn algorithm, the step $f_k \leftarrow g_k^{\Cc, \varepsilon}$ is supposed to be better than any gradient step in the direction of $L^2(\Xc;\alpha)$, yet does not yield sufficient increase under assumption \eqref{eq:assumption}. 
\end{proof}

\begin{proof}[Proof of the error estimate]
Let us analyze of step of the algorithm starting from a potential $f_k$. For shorter notation, write $F_k = \left( f_k^{\Cc, \varepsilon} \right)^{\Cc, \varepsilon}$. Moreover, if the NR step is accepted, we have $f_{k+1} = F_k + \Delta F_k$, where $\Delta F_k = -\left( \nabla_{V_\ell}^2 Q_\varepsilon^{\mathrm{semi}}(F_k) \right)^{-1} \nabla_{V_\ell} Q_\varepsilon^{\mathrm{semi}}(F_k)$ is the NR increment.

Now that the convergence $(f_k, g_k) \stackrel{k \rightarrow \infty}{\longrightarrow} (f^*, g^*)$ is established, Lemma \ref{lemma:equivalent} applies. Upon writing $\K_\varepsilon$ in block form as Eq. \eqref{eq:K_form}, and applying the Lemma twice, we have
\begin{align}
\label{eq:first_error}
    F_k - f^* = 
  & \ \left( f_k^{\Cc, \varepsilon} \right)^{\Cc, \varepsilon} - f^* \\
= & \ -R \left( f_k^{\Cc, \varepsilon} - g^* \right) + \Oc\left( \|f_k - f\|^2 \right) \nonumber \\
= & \ R R^T \left( f_k - f^* \right) + \Oc\left( \|f_k - f\|^2 \right) \nonumber \ .
\end{align}

Since for $k$ large enough, we are in the basin of attraction of Newton-Raphson, the NR step is accepted.  As such, upon substituting, we have
\begin{align*}
  & f_{k+1}-f^*\\
= & F_k + \Delta F_k - f^*\\
= & R R^T \left( f_k - f^* \right) 
  - \left( \nabla_{V_\ell}^2 Q_\varepsilon^{\mathrm{semi}}(F_k) \right)^{-1} \nabla_{V_\ell} Q_\varepsilon^{\mathrm{semi}}(F_k)
  + \Oc\left( \|f_k - f\|^2 \right) \\
= & R R^T \left( f_k - f^* \right) 
  - \left( \nabla_{V_\ell}^2 Q_\varepsilon^{\mathrm{semi}}(F_k) \right)^{-1} \left( \nabla_{V_\ell} Q_\varepsilon^{\mathrm{semi}}(F_k)
  - \nabla_{V_\ell} Q_\varepsilon^{\mathrm{semi}}(f^*) \right)
  + \Oc\left( \|f_k - f\|^2 \right) \ .
\end{align*}
At this stage, let us introduce $(\Pi_{V_\ell}, \Pi_{V_\ell}^\perp )$ which are the (orthogonal) projectors on $V_\ell$ and its orthogonal. By Taylor expansion, and using the fact that $F_k-f^* = \Oc(f_k-f^*)$, this yields
\begin{align*}
  & f_{k+1}-f^*\\
= & R R^T \left( f_k - f^* \right) 
  - \left( \nabla_{V_\ell}^2 Q_\varepsilon^{\mathrm{semi}}(F_k) \right)^{-1} 
  \Pi_{V_\ell} \left( \nabla Q_\varepsilon^{\mathrm{semi}}(F_k)
  - \nabla Q_\varepsilon^{\mathrm{semi}}(f^*) \right)
  + \Oc\left( \|f_k - f\|^2 \right) \\
= & R R^T \left( f_k - f^* \right) 
  - \left( \nabla_{V_\ell}^2 Q_\varepsilon^{\mathrm{semi}}(f^*) \right)^{-1} 
  \Pi_{V_\ell} \nabla^2 Q_\varepsilon^{\mathrm{semi}}(f^*) \left(  F_k - f^* \right)
  + \Oc\left( \|f_k - f\|^2 \right) \ .
\end{align*}
By reinjecting Eq. \eqref{eq:first_error}, we find
\begin{align*}
  & f_{k+1}-f^*\\
= & R R^T \left( f_k - f^* \right) 
  - \left( \nabla_{V_\ell}^2 Q_\varepsilon^{\mathrm{semi}}(f^*) \right)^{-1} 
  \Pi_{V_\ell} \nabla^2 Q_\varepsilon^{\mathrm{semi}}(f^*) R R^T \left(  f_k - f^* \right)
  + \Oc\left( \|f_k - f\|^2 \right) \\
= & \left[ \Id 
  - \left( \nabla_{V_\ell}^2 Q_\varepsilon^{\mathrm{semi}}(f^*) \right)^{-1} 
  \Pi_{V_\ell} \nabla^2 Q_\varepsilon^{\mathrm{semi}}(f^*) 
  \right]
  R R^T \left(  f_k - f^* \right)
  + \Oc\left( \|f_k - f\|^2 \right) \ .
\end{align*}

Now, invoke the fact that
$$
\nabla_{V_\ell}^2 Q_\varepsilon^{\mathrm{semi}}(f^*)
=
\Pi_{V_\ell} \nabla^2 Q_\varepsilon^{\mathrm{semi}}(f^*) \Pi_{V_\ell} \ ,
$$
so that
\begin{align*}
  & f_{k+1}-f^*\\
= & \left[ \Id 
  - \left( \nabla_{V_\ell}^2 Q_\varepsilon^{\mathrm{semi}}(f^*) \right)^{-1} 
  \Pi_{V_\ell} \nabla^2 Q_\varepsilon^{\mathrm{semi}}(f^*) \left( \Pi_{V_\ell} + \Pi_{V_\ell^\perp} \right)
  \right]
  R R^T \left(  f_k - f^* \right)
  + \Oc\left( \|f_k - f\|^2 \right) \\
= & \left[ \Id 
  - \Pi_{V_\ell}
  \left( \nabla_{V_\ell}^2 Q_\varepsilon^{\mathrm{semi}}(f^*) \right)^{-1}
  \Pi_{V_\ell} \nabla^2 Q_\varepsilon^{\mathrm{semi}}(f^*) \Pi_{V_\ell^\perp} 
  \right]
  R R^T \left(  f_k - f^* \right)
  + \Oc\left( \|f_k - f\|^2 \right) \\
 = & \left[ \Pi_{V_\ell^\perp}
  - \left( \nabla_{V_\ell}^2 Q_\varepsilon^{\mathrm{semi}}(f^*) \right)^{-1} 
  \Pi_{V_\ell} \nabla^2 Q_\varepsilon^{\mathrm{semi}}(f^*) \Pi_{V_\ell^\perp}
  \right]
  R R^T \left(  f_k - f^* \right)
  + \Oc\left( \|f_k - f\|^2 \right)\ .
\end{align*}

Recall that we have assumed that $V_\ell = V_\ell(\varepsilon)$ is exactly the eigenspace spanned by the first eigenvalues of $\nabla^2 Q_\varepsilon^{\mathrm{semi}}(f^*)$. In such a case, projectors commute with the Hessian so that the following great simplification occurs. We have
$$
\Pi_{V_\ell(\varepsilon)} \nabla^2 Q_\varepsilon^{\mathrm{semi}}(f^*) \Pi_{V_\ell(\varepsilon)^\perp} = 0 \ ,
$$
and thus
\begin{align*}
  f_{k+1}-f^*
= & \ \Pi_{V_\ell(\varepsilon)}^\perp R R^T \left( f_k - f^* \right) 
  + \Oc\left( \|f_k - f\|^2 \right) \ .
\end{align*}
The operator norm of $\Pi_{V_\ell}^\perp R R^T$ is indeed controlled by $\rho_{\ell+1}^2$, and we are done with the proof. Also notice that the bound is sharp, at first order.
\end{proof}

On a concluding note, it is possible to write an error estimate which takes into account that in practice, $V_\ell$ is not exactly $V_\ell(\varepsilon)$. Quantifying the difference would add a term that depends on $\| \Pi_{V_\ell} - \Pi_{V_\ell(\varepsilon)} \|_{\mathrm{op}} = \| \Pi_{V_\ell^\perp} - \Pi_{V_\ell(\varepsilon)^\perp} \|_{\mathrm{op}}$. We chose to not do that as it only complicates the final statement and the proof. For example the projector $\Pi_{V_\ell}$ no longer commutes $\nabla^2 Q_\varepsilon^{\mathrm{semi}}(f^*)$. Furthermore, the control is not very explicit and would certainly benefit from a better understanding of how $V_\ell(\varepsilon)$ and $V_\ell(\varepsilon')$ differ for $\varepsilon'>\varepsilon$. 

Let us sketch what happens for the sake of completeness. If $V_\ell$ is not assumed to be $V_\ell(\varepsilon)$, i.e., spanned by eigenvectors of the Hessian, then
\begin{align*}
   & f_{k+1}-f^* \\
=  & \ \left[ \Pi_{V_\ell^\perp} - \Pi_{V_\ell(\varepsilon)^\perp}
  - \left( \nabla_{V_\ell}^2 Q_\varepsilon^{\mathrm{semi}}(f^*) \right)^{-1} 
  \Pi_{V_\ell} \nabla^2 Q_\varepsilon^{\mathrm{semi}}(f^*) \Pi_{V_\ell^\perp}
  \right]
  R R^T \left(  f_k - f^* \right) \\
  & \ \quad + \Pi_{V_\ell(\varepsilon)}^\perp R R^T \left( f_k - f^* \right)
  + \Oc\left( \|f_k - f\|^2 \right) \ .
\end{align*}
Upon invoking the triangle inequality and bounding using operator norms, we have
\begin{align*}
   & \| f_{k+1}-f^* \| \\
\leq  & \ \left\| \left( \nabla_{V_\ell}^2 Q_\varepsilon^{\mathrm{semi}}(f^*) \right)^{-1} 
  \Pi_{V_\ell} \nabla^2 Q_\varepsilon^{\mathrm{semi}}(f^*) \Pi_{V_\ell^\perp}
  \right\|_{\mathrm{op}}
  \|  f_k - f^* \| \\
  & \ \quad + \left( \| \Pi_{V_\ell} - \Pi_{V_\ell(\varepsilon)} \|_{\mathrm{op}} + \rho_{\ell+1}^2 \right)
  \|  f_k - f^* \|
  + \Oc\left( \|f_k - f\|^2 \right) \\
\leq  & \ \left\| \left( \nabla_{V_\ell}^2 Q_\varepsilon^{\mathrm{semi}}(f^*) \right)^{-1}
  \right\|_{\mathrm{op}}
  \left\|
  \Pi_{V_\ell} \nabla^2 Q_\varepsilon^{\mathrm{semi}}(f^*) \Pi_{V_\ell^\perp}
  -
  \Pi_{V_\ell(\varepsilon)} \nabla^2 Q_\varepsilon^{\mathrm{semi}}(f^*) \Pi_{V_\ell(\varepsilon)^\perp}
  \right\|_{\mathrm{op}}
  \|  f_k - f^* \| \\
  & \ \quad + \left( \| \Pi_{V_\ell} - \Pi_{V_\ell(\varepsilon)} \|_{\mathrm{op}} + \rho_{\ell+1}^2 \right)
  \|  f_k - f^* \|
  + \Oc\left( \|f_k - f\|^2 \right) \\
\leq  & \ \left\| \left( \nabla_{V_\ell}^2 Q_\varepsilon^{\mathrm{semi}}(f^*) \right)^{-1}
  \right\|_{\mathrm{op}}
  \left(
  2 \| \Pi_{V_\ell} - \Pi_{V_\ell(\varepsilon)} \|_{\mathrm{op}}
  \| \nabla^2 Q_\varepsilon^{\mathrm{semi}}(f^*) \|_{\mathrm{op}}
  \right)
  \|  f_k - f^* \| \\
  & \ \quad + \left( \| \Pi_{V_\ell} - \Pi_{V_\ell(\varepsilon)} \|_{\mathrm{op}} + \rho_{\ell+1}^2 \right)
  \|  f_k - f^* \|
  + \Oc\left( \|f_k - f\|^2 \right) \\
= & \ \left[ \rho_{\ell+1}^2 
  + c
  \| \Pi_{V_\ell} - \Pi_{V_\ell(\varepsilon)} \|_{\mathrm{op}}
  \right]
  \|  f_k - f^* \|
  + \Oc\left( \|f_k - f\|^2 \right) \ ,
\end{align*}
where 
$$
c := 1 + 2 \left\| \left( \nabla_{V_\ell}^2 Q_\varepsilon^{\mathrm{semi}}(f^*) \right)^{-1}
  \right\|_{\mathrm{op}}
  \| \nabla^2 Q_\varepsilon^{\mathrm{semi}}(f^*) \|_{\mathrm{op}}
  \ .
$$
Notice that this bound specializes to our error estimate, when $\| \Pi_{V_\ell} - \Pi_{V_\ell(\varepsilon)} \|_{\mathrm{op}} = 0$. It is however unclear whether the constant $c$ is sharp.